\documentclass{article}
\usepackage[utf8]{inputenc}
\usepackage{amssymb}
\usepackage{amsthm}
\usepackage{amsfonts,amsmath,amsbsy,latexsym,mathrsfs
}
\usepackage{authblk}

\usepackage{enumerate}
\usepackage{amsmath}
\usepackage{amsthm}
\usepackage{amssymb}
\usepackage{amsbsy}
\usepackage{amsfonts}
\usepackage{amstext}
\usepackage{amscd}
\usepackage{tikz}
\newtheorem{thm}{Theorem}
\newtheorem{cor}[thm]{Corollary}
\newtheorem{lem}[thm]{Lemma}
\newtheorem{prop}[thm]{Proposition}
\newtheorem{defi}[thm]{Definition}
\newtheorem{rem}[thm]{Remark}
\newtheorem{ejem}[thm]{Example}

\title{Energy  and Randic index of directed graphs}

\author{Gerardo Arizmendi and Octavio Arizmendi\thanks{O.A. received support from Conacyt Grant CB-2017-2018-A1-S-9764 an from the European Union's Horizon 2020 research and innovation programme under the Marie Sk\l{}odowska-Curie grant agreement No 734922. He thanks Roland Speicher and Moritz Weber for the nice environment at University of Saarlands, while visiting for a sabbatical period. \\~}}

\date{ \today}

\begin{document}

\maketitle

\begin{abstract}
The concept of Randic index has been extended recently for a digraph. We prove that  $2R(G)\leq \mathcal{E}(G)\leq 2\sqrt{\Delta(G)} R(G)$, where $G$ is a digraph, and  $R(G)$ denotes the Randic index, $\mathcal{E}(G)$ denotes the Nikiforov energy and $\Delta(G) $ denotes the maximum degree of $G$. In both inequalities we describe the graphs for which the equality holds.
\end{abstract}


\section{Introduction}


The energy of graph and the Randic index are well known graph invariants defined from considerations in chemical graph theory. This two quantities are known to be good descriptors of any graph and their properties have been explored in networks. Since for many networks the relation between nodes is non-symmetric it is natural to study these descriptors for directed graphs, where, for example, the inner and outer degrees play very different roles in terms of the ``flow" structure of the network.

The following inequalities have been proven in \cite{AyA} and \cite{YLPL}:
\begin{equation}\label{ME}
2R(G)\leq \mathcal{E}(G)\leq 2\sqrt{\Delta(G)} R(G),
\end{equation}
where $G$ is an (undirected) simple graph, $R(G)$ the Randic index of $G$, $\mathcal{E}(G)$ the energy of $G$ and $\Delta(G)$ the maximum degree of $G$. A notion of energy of a digraph was defined in \cite{Nik}. Recently, the definition of Randic index  has been extended and studied for digraphs \cite{BMR,MR}. The main purpose of this paper is to extend the inequality in (\ref{ME}) to directed graphs. This is proved in Theorem \ref{mainT2} and Theorem \ref{mainT3}. In this case, since the adjacency matrix of a directed graph is not symmetric in general and our definitions of energy and Randic index change, we need adapt our methods. In particular, we define inner and outer energies which we relate with the inner and outer degrees.

In order to describe the graphs for which the equalities in (\ref{ME}) hold we use, what we call, the Hermitianization trick, which relates the energy of a digraph with the energy of a bipartite graph. Moreover, this technique allows to give another proof to 
Theorems \ref{mainT2} and \ref{mainT3}.

Apart from this introduction, the paper is organized as follows. In Section \ref{secenergy}, we introduce the energy of a digraph as defined by Nikiforov. We also define the outer energy of a vertex $\mathcal E^+(v)$ and inner energy of a vertex $\mathcal E^-(v)$ and prove that for adjacent vertices ${\cal E}^+(v_i){\cal E}^-(v_j)\geq1$. In Section \ref{secRE}, we prove the main results of the paper, namely that inequalities in ($\ref{ME}$) are satisfied for digraphs and their corresponding Randic index and energy. Section \ref{secH} is devoted to the Hermitianization trick. We use this technique to give another proof of the main theorems of the paper and to describe the graphs in which the equalities in (\ref{ME}) are fullfilled.





\section{Digraph energy and vertex energy}\label{secenergy} 

A (finite) directed graph or digraph is a pair $G=(V,E)$ where $V$ is a finite set and $E\subset V\times V$, the elements in $V$ are called vertices and the elements in $E$ are called edges. We assume that $G$ is simple, i.e.  $(v,v)\notin E$ for all $v\in V$. For a digraph $G$ with $n$ vertices, the adjacency matrix,  denoted by $A=A(G)$, is the $n\times n$ matrix with entries 
$A_{ij}=1$ if $(v_i,v_j)\in E$ and $A_{ij}=0$ otherwise. 


The energy of a graph is given by the trace of the absolute value of its adjacency matrix \cite{Gut,GR,LSG}. In analogy to graphs, the energy of a digraph is defined as follows.  Let us consider a digraph $G=(V,E)$ with adjacency matrix $A\in M_n(\mathbb{R})$. One can define the absolute value of 
$A$ in two ways:  $$|A|^+=(AA^t)^{1/2}$$
$$|A|^-=(A^tA)^{1/2}.$$
both definitions coincide when $A$ is symmetric, i.e. $G$ is undirected. 

\begin{defi}
The energy of a digraph $G$, denoted  $\mathcal{E}(G)$, is given by 
$$\mathcal{E}(G)=Tr(|A|^+)=Tr(|A|^-)$$
where A is the adjacency matrix of $G$.
\end{defi}
Note that, $$\mathcal{E}(G)=\sum^n_{i=1} \sigma_i$$ where $(\sigma_i)^n_{i=1}$ denote the set of singular values of $A(G)$, counted with multiplicities.

In \cite{AJ}, the energy of a vertex is introduced. We extend this definition to digraphs.
\begin{defi}
The outer energy of the vertex $v_i$ with respect to $G$, which is denoted by $\mathcal{E}^+_G(v_i)$, is given by
\begin{equation}
  \mathcal{E}^+_G(v_i)=|A|^+_{ii}, \quad\quad~~~\text{for } i=1,\dots,n,
\end{equation}
where $|A|^+=(AA^t)^{1/2}$ and $A$ is the adjacency matrix of $G$.
\end{defi}

\begin{defi}
The inner energy of the vertex $v_i$ with respect to $G$, which is denoted by $\mathcal{E}^-_G(v_i)$, is given by
\begin{equation}
  \mathcal{E}^-_G(v_i)=|A|^-_{ii}, \quad\quad~~~\text{for } i=1,\dots,n,
\end{equation}
where $|A|^-=(A^tA)^{1/2}$ and $A$ is the adjacency matrix of $G$.
\end{defi}

In this way the energy of a graph is given by the sum of the individual energies of the vertices of $G$,
\begin{equation*}
  \mathcal{E}(G)=\mathcal{E}^+_G(v_1)+\cdots+\mathcal{E}^+_G(v_n)=\mathcal{E}^-_G(v_1)+\cdots+\mathcal{E}^-_G(v_n).
\end{equation*}

It is important to remark, as observed by \cite{Benzi} that, in general, the inner and outer energy of a vertex do not coincide.

\subsection{Energy of adjacent vertices} 

The next theorem is fundamental for the proof of our main result.

\begin{thm}\label{T1}
Let $v_i$ and $v_j$ be connected vertices of a simple digraph G. Then ${\cal E}^+(v_i){\cal E}^-(v_j)\geq1$. 
\end{thm}
\begin{proof}

 Let $A(G)$ be the adjacency matrix of $G$. Then, using the singular decomposition we can write $A(G)=UDV^t$, where $U=(u_{kl})$, $V=(v_{kl})$ are orthogonal and $D=(d_{kl})$ is a diagonal matrix. Let us denote $d_{kk}=\lambda_k$. From this, we have that $AA^t=UD^2U^t$ and $AA^t=VD^2V^t$, hence  
 $$|A|^+=U|D|U^t,$$
 $$|A|^-=V|D|V^t.$$
 
 A direct calculation shows that  ${\cal E}^+(v_i)=\sum_ku_{ik}^2|\lambda_k|$ and ${\cal E}^-(v_j)=\sum_kv_{jk}^2|\lambda_k|$. Moreover $A(G)_{ij}=\sum_ku_{ik}v_{jk}\lambda_k$. Since $v_i$ and $v_j$ are connected then $A(G)_{ij}=1$.

Now consider $$u=(u_{i1}\sqrt{|\lambda_1|},\dots,u_{in}\sqrt{|\lambda_n|})$$ 
and
$$v=(v_{j1}sign(\lambda_1)\sqrt{|\lambda_1|},\dots,v_{jn}sign(\lambda_n)\sqrt{|\lambda_n|})$$ 
then 
$$\left<v,w\right>^2=(\sum_ku_{ik}v_{jk}\lambda_k)^2=1$$
$$||v||^2=\sum_ku_{ik}^2|\lambda_k|={\cal E}^+(v_i)$$
$$||w||^2=\sum_kv_{jk}^2|\lambda_k|={\cal E}^-(v_j)$$
which proves the assertion by the Cauchy-Schwarz inequality.
\end{proof}


By the use of AM-GM inequality we observe that:
\begin{cor}
Let $v_i$ and $v_j$ be connected vertices of a simple digraph $G$. Then ${\mathcal E}^+(v_i)+{\mathcal E}^-(v_j)\geq2$.
\end{cor}

\section{Randic index and energy of digraphs}\label{secRE}


Let $G=(V,E)$ be a digraph. For a vertex $v\in V$, we denote by $d^+(v)$ and $d^{-}(v)$, the outer and inner degrees of $v$, respectively,  given by the cardinality of the sets
$$N^+(v)=\{w\in V| (v,w)\in E\}\mbox{ and } N^-(v)=\{w\in V| (v,w)\in E\}.$$

We denote by $a$ the number of edges (or arcs), i.e $a=|E|$. Observe that $\sum_{v\in V} d^+(v)=a=\sum_{v\in V} d^-(v)$. We denote by $\Delta(G)$ the maximum, over the vertices, of the inner and outer degrees.

 For a digraph $G=(V,E)$, the Randic index was defined in \cite{MR} and is given by
$$R(G)=\frac{1}{2}\sum_{(v,w)\in E} \frac{1}{\sqrt{{d^+(v)d^-(w)}}}.$$
where $d^+(v)$ and $d^-(v)$ correspond to the outer and inner degrees of $v$.

\begin{thm}\label{mainT2}
Let $G$ be a digraph with  energy $\mathcal{E}(G)$ and Randic index $R(G)$ then $\mathcal{E}(G)\geq 2R(G)$.  
\end{thm}

\begin{proof}
Let $G=(V,E)$. For an edge $e=(v,w)$ define $\mathcal{E}(e)=\mathcal{E}^+(v)/deg^+(v)+\mathcal{E}^-(w)/deg^-(w)$.
Then, on one hand, \begin{eqnarray*}
\sum_{e\in E}\mathcal{E}(e)&=&\sum_{e\in E}\left( \frac{\mathcal{E}^+(v)}{deg^+(v)}+\frac{\mathcal{E}^-(w)}{deg^-(w)}\right)\\&=&\sum_{(v,w)\in E} \frac{\mathcal{E}^+(v)}{deg^+(v)}+\sum_{(v,w)\in E}\frac{\mathcal{E}^-(w)}{deg^-(w)}
\\&=&\sum_{v\in V}\sum_{v\sim w} \frac{\mathcal{E}^+(v)}{deg^+(v)}+ \sum_{w\in V(G)}\sum_{w\sim v} \frac{\mathcal{E}^-(w)}{deg^-(w)}\\
&=&\sum_{v\in V} \mathcal{E}^+(v)+\sum_{w\in V} \mathcal{E}^-(w)
=2\mathcal{E}(G).
\end{eqnarray*}

On the other hand, by the classical AM-GM inequality, if $e=(v,w)$,
$$\mathcal{E}(e):=
\frac{\mathcal{E}^+(v)}{deg^+(v)}+\frac{\mathcal{E}^-(w)}{deg^-(w)}\geq 2\sqrt{\frac{\mathcal{E}^+(v)\mathcal{E}^-(w)}{deg^+(v)deg^-(w)}}\geq2 \frac{1}{\sqrt{{deg^+(v)deg^-(w)}}},$$
where we used Theorem \ref{T1} in the last inequality.

Finally, summing over $e\in E(G)$ we obtain the desired inequality

$$\mathcal{E}(G)=\frac{1}{2}\sum_{e\in E(G)} \mathcal{E}(e)\geq  \sum_{(v,w)\in E(G)} \frac{1}{\sqrt{{d^+(v)d^-(w)}}}=2R(G)$$

\end{proof}

Now we prove and analogous result to the one in \cite{AFJ} for digraphs

\begin{lem}
For a digraph $G$ and a vertex $v_i\in G$
$$\mathcal{E}^+(G)(v_i) \leq\sqrt{d^+(i)}$$
and
$$\mathcal{E}^-(G)(v_i) \leq \sqrt{d^-(i)}$$
\end{lem}
We only prove the first inequality, the other one is analogous and also follows by taking $\tilde G$ which is the graph with the same edges as $G$ but with the directions reversed, i.e. $(v,w)\in G$ if and only if $(w,v)\in \tilde G$. 

Let $A(G)$ be the adjacency matrix of $G$. Then, using the singular decomposition, we can write $A(G)=UDV^t$, where $U=(u_{kl})$ and $V=(v_{kl})$ are orthogonal matrices and $D=(d_{kl})$ is  diagonal. Let us denote $d_{kk}=\lambda_k$. From this, we have that $AA^t=UD^2U^t$ and $|A|^+=UDU^t$. Again,  ${\cal E}^+(v_i)=\sum_ku_{ik}^2|\lambda_k|$. Observe now that the elements of the diagonal in $AA^t$ give us the outer degrees of $G$ hence $d^+(i)=(AA^t)_{ii}=\sum_ku_{ik}^2\lambda_k^2$. 

Now consider the vectors $v=(u_{i1}|\lambda_i|,\dots,u_{in}|\lambda_n|)$ and $w=(u_{i1},\dots,u_{in})$ then we have that $${\cal E}^+(v_i)^2=(\sum_ku_{ik}^2|\lambda_k|)^2\leq \sum_ku_{ik}^2\lambda_k^2\sum_ku_{ik}^2=\sum_ku_{ik}^2\lambda_k^2=d^+(v_i)$$ where the inequality follows by the Cauchy-Schwarz inequality for $v$ and $w$.




We now have a  McClelland's type inequality.

\begin{cor}
For a simple digraph $G$ with $n$ vertices and $a$ edges

$$\mathcal{E}(G)\leq \sum_{i=1}^n\sqrt{d^{\pm}_i}\leq\sqrt{an} $$
\end{cor}
\begin{proof}
The first inequality follows by summing the outer or inner energies and comparing each one with the respective degrees. The second one
follows using QM-AM inequality since $\sum d^{\pm}_i=a$.
\end{proof}



\begin{thm}\label{mainT3}
Let $G$ be a digraph with  energy $\mathcal{E}(G)$, Randic index $R(G)$ and maximum degree $\Delta(G)$ then $\mathcal{E}(G)\leq 2\sqrt{\Delta(G)}R(G)$.  
\end{thm}
 \begin{proof}
 
 \begin{eqnarray*}
  \mathcal{E}(G)&=&\frac{1}{2}\sum_{e\in E(G)}\mathcal{E}(e)\\
&=&\frac{1}{2}\sum_{(v,w)\in E(G)}\frac{\mathcal{E}^+(v)}{d^+(v)}+\frac{\mathcal{E}^-(w)}{d^-(w)}\\
&\leq& \frac{1}{2}\sum_{(v,w)\in E(G)}\frac{\sqrt{d^+(v)}}{d^+(v)}+\frac{\sqrt{d^-(w)}}{d^-(w)}\\
&=& \frac{1}{2}\sum_{(v,w)\in E(G)}\frac{1}{\sqrt{d^+(v)}}+\frac{1}{\sqrt{d^-(w)}}\\
&=& \frac{1}{2}\sum_{(v,w)\in E(G)}\frac{\sqrt{d^+(v)}+\sqrt{d^-(v)}}{\sqrt{d^+(v)d^-(w)}}\\
&\leq& (2\sqrt{\Delta(G)})\left(\frac{1}{2}\sum_{(v,w)\in E(G)}\frac{1}{\sqrt{d^+(v)d^-(w)}}\right)\\
&=&(2\sqrt{\Delta(G)})R(G)
 \end{eqnarray*}
 \end{proof}

\section{Hermitianization trick} \label{secH}

In this section we use the relations between different energies as observed in \cite{AJ}.

\subsection{Nikiforov's energy and $(n,m)$-bipartite graphs}
\label{bipartiteNik}

Recall  \cite{Nik}, that for a matrix $M$ the (Nikiforov) energy of $M$, $\mathcal{N}(M)$, is given by $$Tr(\sqrt{MM^T})=\sum^n_i\sigma_i(M),$$
where $(\sigma_i(M))^n_{i=1}$ denotes the set of singular values of $M$.

Now, the adjacency matrix $A$ of a bipartite graph whose parts have $r$ and $s$ vertices has the form
\begin{equation}\label{bipartite}
A =
\begin{pmatrix}
{\bf 0}_{r,r} & M \\
 M^T    & {\bf 0}_{s,s}
\end{pmatrix},
\end{equation}
where $M$ is a $(0,1)$-matrix of size $r \times s$ , and ${\bf 0}$ represents the zero matrix. Conversely, given a $(0,1)$-matrix $M$ of size $r\times s$, the matrix $A$ as in (\ref{bipartite}) is the adjacency matrix of a bipartite graph. Thus there is a one-to-one correspondence between $(r,s)$-bipartite graphs and $(0,1)$-matrices of size $r\times s$.

Moreover, it easily follows from the singular value decomposition that for any matrix $M$, the non-zero eigenvalues of the matrix $A$  in (\ref{bipartite}) are the nonzero singular values of $M$ together with their negatives and thus for any bipartite graph $G$ with adjacency matrix $A$ as in (\ref{bipartite}), one has
\begin{equation} \label{bipartite-Nikiforov}
   \mathcal{E}(G)=2\mathcal{N}(M).
\end{equation}

In other words, studying the energy of $(r,s)$-bipartite graphs corresponds to studying the Nikiforov's energy of $(0,1)$-matrices of size $r\times s$. 

\subsection{Digraphs and $(n,n)$-bipartite graphs.}

We are interested in the particular case when $r=s=n$. 

Let $G=(V,E)$ be a directed graph on $V=\{1,\dots,n\}$. We denote by $B(G)=(\overline{V},\overline{E})$ the undirected bipartite graph with vertex set $\overline{V} =\{1^{-},\dots,n^{-},1^{+},\dots,n^{+}\}$ and edge set described as follows: $(i,j)\in E$ if and only if $\{i^-,j^+\}\in \overline{E}$, i.e $i\to j\iff i^-\sim j^+$. 
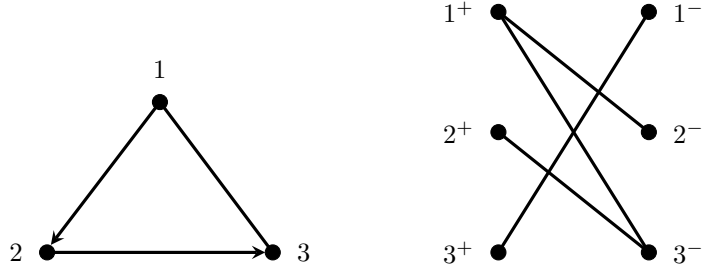
\begin{figure}[h]

\begin{center}

\begin{tikzpicture}[scale=2]

\draw [>=stealth, ->,very thick] (0-3,0) -- (1.45-3,0);
\draw [>=stealth, ->,very thick] (.75-3,1) --(.025-3,.05) ;
\draw [ very thick] (.75-3,1) -- (1.5-3,0);

\node[left] at (-0.1-3,0)  {$2$};
\node[right] at (1.6-3,0)  {$3$};
\node[above] at (.75-3,1.1)  {$1$};

\draw[fill] (0-3,0) circle [radius=0.05];
\draw[fill] (1.5-3,0) circle [radius=0.05];
\draw[fill] (.75-3,1) circle [radius=0.05];

\draw [very thick] (0,1.6) -- (1,.8);
\draw [very thick] (0,.8) -- (1,0);
\draw [very thick] (0,1.6) -- (1,0);
\draw [very  thick] (0,0) -- (1,1.6);

\node[left] at (-0.1,0)  {$3^+$};
\node[left] at (-0.1,.8)  {$2^+$};
\node[left] at (-0.1,1.6)  {$1^+$};
\node[right] at (1.1,0)  {$3^-$};
\node[right] at (1.1,.8)  {$2^-$};
\node[right] at (1.1,1.6)  {$1^-$};

\draw[fill] (0,0) circle [radius=0.05];
\draw[fill] (0,.8) circle [radius=0.05];
\draw[fill] (0,1.6) circle [radius=0.05];
\draw[fill] (1,0) circle [radius=0.05];
\draw[fill] (1,.8) circle [radius=0.05];
\draw[fill] (1,1.6) circle [radius=0.05];

\end{tikzpicture}

\caption{\label{fig3}  A directed graph and its corresponding bipartite graph.}

\end{center}

\end{figure}

Moreover, the relation between the adjacency matrices of these graphs is as in \eqref{bipartite}. Thus, \eqref{bipartite-Nikiforov}, we obtain a direct relation between their energies.

\begin{prop} \label{PropEn}
For any directed graph $G$,
   \begin{equation}
   2\mathcal{E}(G)=\mathcal{E}(B(G)).
\end{equation}
\end{prop}

In order to prove our main theorem we need to relate the Randic index of $G$ with the Randic index of $B(G)$, which amounts

\begin{lem}
For any directed graph $G$ with vertex set $[n]=\{1,\dots,n\}$, and $B(G)$ with vertex set $\{1^-,\dots,n^-,1^+,\dots,n^+\}$, the following relations holds, for any $i\in[n]$:
$$d^+(i)=d(i^+), \qquad d^-(i)= d(i^-),$$
where $d^+:V\to \mathbb{N}$,$d^-=V\to \mathbb{N}$ denote the out-degree and in-degree in G, and $d:\overline{V}\to \mathbb{N}$ denotes de degree in $B(G)$. 
\end{lem}
\begin{proof}
This is clear by taking cardinalities in set relations
$$|\{j\in[n] |i\to j\}|=|\{j\in[n]|i^-\to j^+\}|$$
and 
$$|\{i\in[n]|i\to j\}|=|\{i\in[n]|i^-\to j^+\}|.$$
\end{proof}

From the above simple lemma one then sees that $2R(G)=R(B(G))$.

\begin{prop} \label{PropRand}
For any directed graph $G$
   \begin{equation} 
   2R(G)=R(B(G)).
\end{equation}
\end{prop}

\begin{proof}
Since $(i,j)\in E(G)$ exactly when $\{i^-,j^+\}\in E(B(G))$, and any edge in $E(B(G))$ is of the form $\{i^-,j^+\}$, for some $i,j\in[n]$, then
\begin{eqnarray} 2R(G)&=&\sum_{(i,j)\in E(G)}\frac{1}{\sqrt{d^+(i)d^-(j)}}\\
&=&\sum_{\{i^-j^+\}\in E(B(G))}\frac{1}{\sqrt{d(i^+)d(j^-)}}\\
&=&R(B(G)).
\end{eqnarray}
\end{proof}



Proofs of Theorems \ref{mainT2}  and \ref{mainT3} now follow from Propositions \ref{PropEn} and \ref{PropRand}, together with equation \eqref{ME}, applied to $B(G)$.

Given a digraph $G=(V,G)$, a vertex $v\in G$ is called a sink if $d^+(v)=0$ and is called a source if $d^-(v)=0$.
We say that $G$ is a sink-source digraph if every vertex of $G$ is either a sink or a source.

A sink-source digraph can be thought as a bipartite graph since we can split the set of 
vertices $V$ into two subsets $V_1, V_2$ and such that $ij\in E$ if and only if $i\in V_1$ and $j \in V_2$,
moreover if $G$ there is a two to one correspondence between weakly connected sink-source digraphs and connected bipartite graphs given by choosing the sink and the source sets.

\begin{defi}
A splitting of a digraph $G$ into sink-source digraphs is a set $\{G_1,\dots, G_n\}$ of digraphs
such that  $G=\cup_{i=1}^nG_i$ with the condition that  $d^{\pm}_{G_i}(v)\neq0$ implies $d^{\pm}_G(v)=d^{\pm}_{G_i}(v)$.
\end{defi} 

 \begin{ejem} Consider the graph $G$ with vertex set $V=\{1,2,3,4,5\}$ and edge set  $E=\{(1,4),(1,5), (2,4), (5,1),(5,2), (5,3)\}$
 then $G$ can be split into sink-source graphs $G_1=(V_1,E_1)=(\{1,2,4,5\},\{(1,4),(1,5), (2,4)\})$ and $G_2=(V_2,E_2)=
 (\{1,2,3,5\},\{(5,1),(5,2), (5,3)\})$, see Figure \ref{sinkdec}.

 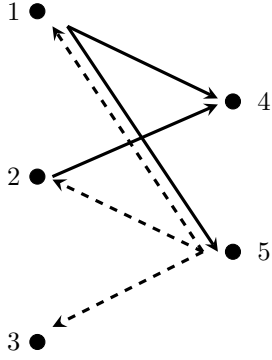
\begin{figure} [h]
    
    \begin{center}

\begin{tikzpicture}[scale=2]
\draw [very thick,->, >=stealth] (.1,.5) -- (1.1,0.02);
\draw [ very thick,->, >=stealth] (.1,.5) -- (1.1,-1);
\draw [ very thick,->, >=stealth] (0,-.5) -- (1.1,-0.02);
\draw [dashed , dashed,very thick,->, >=stealth] (1,-1) -- (0,.5);
\draw [dashed,very thick,->, >=stealth] (1,-1) -- (0,-.52);
\draw [dashed,very thick,->, >=stealth] (1,-1) -- (0,-1.5);
\node[left] at (-.15,.6)  {$1$};
\node[left] at (-.15,-.5)  {$2$};
\node[left] at (-.15,-1.6)  {$3$};

\node[right] at (1.3,0)  {$4$};
\node[right] at (1.3,-1)  {$5$};

\draw[fill] (-.1,.6) circle [radius=0.05];
\draw[fill] (-.1,-.5) circle [radius=0.05];
\draw[fill] (-.1,-1.6) circle [radius=0.05];
\draw[fill] (1.2,0) circle [radius=0.05];
\draw[fill] (1.2,-1) circle [radius=0.05];
\end{tikzpicture}
\end{center}
 \caption{Decomposition of a graph into sink-source graphs}
     \label{sinkdec}
 \end{figure}

 \end{ejem}

\begin{rem} Not every digraph admits a splitting into sink-source digraphs. The graph with $V=\{1,2,3\}$ and
$E=\{(1,2),(1,3),(2,3)\}$ doesn't admit a splitting into sink-source digraphs.
\end{rem}

\begin{thm} Let $G$ be a digraph, then
 $\mathcal{E}(G)= 2R(G)$ if and only if there is a splitting of $G$ into sink-source digraphs $\{G_1,\dots, G_k\}$ such that
 $G_i=\overrightarrow{K}_{n_i,m_i}$, for some $n_i,m_i\in\mathbb{N}$.
\end{thm}

\begin{proof}

 Given a digraph $G$, by using the Hermitianization trick then  $\mathcal{E}(G)= 2R(G)$ if and only if
 $\mathcal{E}(B(G))= 2R(B(G))$. Suppose then that the equality is satisfied, then by Theorem 6 in \cite{AyA}, this implies that $B(G)$ is the disjoint union of complete bipartite graphs. 
 Each complete bipartite  graph ${K}_{n_i,m_i}\subset B(G)$ corresponds to a graph $\overrightarrow{K}_{n_i,m_i}\in
 G $. For each source vertex  $v\in \overrightarrow{K}_{n_i,m_i}$, the fact that the ${K}_{n_i,m_i}$'s are disjoint implies that
$d^+_G (v)=d_{B(G)}(v^+)=n_i$, analogously, for each sink vertex $v\in \overrightarrow{K}_{n_i,m_i}$ we have that
$d^-_G (v)=d_{B(G)}(v^-)=m_i$. Hence $\{\overrightarrow{K}_{n_1,m_1},\dots,\overrightarrow{K}_{n_l,m_k}\}$ is splitting into sink-source digraphs. 

Now, suppose that there is a splitting of $G$ into sink-source digraphs, that is $G$ is the disjoint union $\{\overrightarrow{K}_{n_1,m_1},\dots,\overrightarrow{K}_{n_l,m_k}\}$, and suppose that $v\in \overrightarrow{K}_{n_i,m_i}$, is a source, then  $n_i=d^+_G(v)$, similarly for each sink  $w\in \overrightarrow{K}_{n_i,m_i}$ we have that 
$m_i=d^-_G(w)$. Let now $V_i=\{v^i_1,\dots,v^i_{n_i}\}$ denote the set of sources of $\overrightarrow{K}_{n_i,m_i}$ and  $W^i\{w^i_1,\dots,w^i_{m_i}\}$ the set of sinks, since $n_i=d^+_G(v)$ then  $(v^i_j,w)\in E(G)$ implies that $w=w^i_k$ for some $k\in\{1,\dots,m_i\}$, analogously, since  $m_i=d^-_G(w)$, then $(v,w^i_j)\in E(G)$ implies that $v=v^i_k$ for some $k\in\{1,\dots,n_i\}$, that is, the edges that start in $V^i$ always end in $W^i$ and the edges that
end in $W^i$ always start in $V^i$. The corresponding graph in $B(G)$ is a complete bipartite graph $K_{n_i,m_i}$ with the set $V^+_i=\{(v^i)^+_1,\dots,(v^i)^+_{n_i}\}$ connected to the set  $W^-_i=\{(w^i)^-_1,\dots,(w^i)^-_{m_i}\}$ which is disjoint with other graphs since, the last statement implies that all the elements of $V^+_i$ are only connected with the elements of $W^-_i$ and all the elements of $W^-i$ are only connected with the elements of $V^+_i$. Consequently, $B(G)$ is a union of complete bipartite graphs and
 $\mathcal{E}(B(G))= 2R(B(G))$. This completes the proof.

 \end{proof}

\begin{lem}\label{lemCP}
Let $G$ be a weakly connected finite simple digraph such that for all vertices
$d^{\pm}(v)\leq 1$ then $G$ is either a directed path or a directed cycle.
\end{lem}

\begin{proof}

Let $n$ is the number of elements in $G$. If $n=1$ then $G$ is an isolated vertex. 

When $n>1$ we divide into two cases.

\textbf{Case 1}. There is a vertex $v_1$ with $d^-(v_1)=0$. Then $d^+(v)=1$ which means that there is a unique element $v_2$ such that $(v_1,v_2)$ is an edge, if $n>2$ then $d^+(v)=1$ since, if this is false, we will have that $\{v_1,v_2\}$ is a weakly connected component of $G$ because there cannot be other edges pointing $v_2$ by the fact that $d^-(v_2)\leq 1$. So there is a unique element $v_3$ such that $(v_2,v_3)$ is an edge. Inductively, there should be a sequence of vertices $v_1,\dots,v_n$ such that $(v_i,v_{i+1})\in E(G)$ for all $i\in\{1,\dots, n-1\}$. This means $G$ is a directed path.

\textbf{Case 2.} There is a vertex $v_1$ with $d^-(v_1)=0$. The same proof as case 1 works with the obvious modifications.

\textbf{Case 3.} For every vertex $v\in V$, $d^+(v)=d^-(v)=1$. Take now any vertex $v_1\in V$ then since for every vertex $d^+(v)=1$ we can form unique sequence of vertices $\{v_i\}_{i\in \mathbb N}$  such that $(v_i,v_{i+1})\in E(G)$ since the graph is finite this sequence should repeat. That is, there is some $k$ and $l$ such that $v_{k+l}=v_{l}$. Now, consider the directed path $v_l\to v_{l+1}\to v_{l+k}=v_l$. For each $v_i$, since $d^+(v)=d^-(v)=1$, there is a unique $w$ such that $w,v_i\in E(G)$, which should be $v_{i-1}$ and similarly, $v_{i+1}$ is the unique vertex after $v_i$.  This means that $\{v_l,\dots,v_{k+1}\}$ is directed cycle and a weakly connected component. Thus $k=n$ and $G$ must be this directed cycle.
\end{proof}

\begin{thm}
Let $G$ be a digraph with  energy $\mathcal{E}(G)$ and Randic index $R(G)$ and maximum degree $\Delta(G)$ then $\mathcal{E}(G)= 2\sqrt{\Delta(G)}R(G)$ if and only if $G$ is the disjoint union of directed cycles $\overrightarrow C_n$, directed paths  $\overrightarrow P_n$ and isolated vertices. 
\end{thm}
 \begin{proof}
 
 Suppose that $G$ is a disjoint union of directed paths  or directed cycles then since for each edge $(i,j)$ $d^+(i)=d^-(j)=1$ then $R(G)=1/2a$ where $a$  is the number of arcs. A direct calculation shows that $\mathcal E(\overrightarrow C_n)=n$  and
$\mathcal E(\overrightarrow P_n)=n-1$, which coincides with their number of arcs.

Now suppose that the equality holds. Since $\Delta(G)=\Delta(B(G))$ and $2R(G)=R(B(G))$ and $2\mathcal{E}(G)=\mathcal{E}(B(G))$ then the equality holds if and only if it holds for $B(G)$. Using Theorem 16 in  \cite{YLPL} this is equivalent to
$B(G)$ being the disjoint union of the path $P_2$ and isolated vertices.  Since all degrees in $B(G)$ are either $0$ or $1$ this implies that all in-degrees and out-degrees of  $G$ are also $0$ or $1$. The use of Lemma \ref{lemCP} finishes the proof.

 \end{proof}

\noindent Department of Actuarial Sciences, Physics and Mathematics. Universidad de las Am\'{e}ricas Puebla. San Andr\'{e}s Cholula, Puebla. M\'{e}xico.\\\noindent Email: \emph{gerardo.arizmendi@udlap.mx}
	\\~
	
\noindent Department of Probability and Statistics. Centro de Investigaci\'{o}n en Matem\'aticas, Guanajuato, M\'{e}xico. \\\noindent Email: \emph{octavius@cimat.mx}
\end{document}